\documentclass[10pt]{amsart}

\usepackage{amsfonts}
\usepackage{ifthen}
\usepackage{amsthm}
\usepackage{amsmath}
\usepackage{graphicx}
\usepackage{amscd,amssymb,amsthm}
\usepackage{graphicx}
\usepackage{epstopdf}

\newcommand{\lp}{\mathcal{LP}}
\newcommand{\lpi}{\mathcal{LP}I}
\newcommand{\g}{\gamma}

\newtheorem{lemma}{Lemma}
\newtheorem{theorem}{Theorem}

\newcommand{\real}{\operatorname{Re}}


\begin{document}
\title[Radii of starlikeness and convexity of some $q$-Bessel functions]{Radii of starlikeness and convexity of some $q$-Bessel functions}

\author[\'A. Baricz]{\'Arp\'ad Baricz}
\address{Department of Economics, Babe\c{s}-Bolyai University, Cluj-Napoca 400591, Romania}
\address{Institute of Applied Mathematics, John von Neumann Faculty of Informatics, \'Obuda University, 1034 Budapest, Hungary}
\email{bariczocsi@yahoo.com}

\author[D.K. Dimitrov]{Dimitar K. Dimitrov}
\address{Departamento de Matem\'atica Aplicada, IBILCE, Universidade Estadual Paulista UNESP, S\~{a}o Jos\'e do Rio Preto 15054, Brazil}
\email{dimitrov@ibilce.unesp.br}

\author[I. Mez\H o]{Istv\'an Mez\H o}
\address{Department of Mathematics, Nanjing University of Information Science and Technology, 5 Panxin Rd, Pukou, Nanjing, Jiangsu, P.R. China}
\email{istvanmezo81@gmail.com}

\thanks{$\bigstar$The research of \'A. Baricz is supported by the Romanian National Authority for Scientific Research, CNCS-UEFISCDI, under Grant
PN-II-RU-TE-2012-3-0190. The work of \'A. Baricz was completed during his visit in August 2014 to Department of Mathematics of Nanjing University of Information Science and Technology, to which this author is grateful for hospitality. The research of I. Mez\H o is supported by the Scientific Research Foundation of Nanjing University of Information Science and Technology. The research of D. K. Dimitrov is supported by the Brazilian foundations CNPq under Grant 307183/2013--0 and FAPESP under Grants 2009/13832--9.}

\maketitle

\begin{center}
Dedicated to Professor Mourad E. H. Ismail on the occasion of his 70th birthday
\end{center}

\begin{abstract}
Geometric properties of the Jackson and Hahn-Exton $q$-Bessel functions are studied.
For each of them, three different normalizations are applied in such a way that
the resulting functions are analytic in the unit disk of the complex plane.
For each of the six functions we determine the radii of starlikeness and convexity precisely
by using their Hadamard factorization. These are $q$-generalizations of some known results for
Bessel functions of the first kind. The characterization of entire functions from the
Laguerre-P\'olya class via hyperbolic polynomials play an important role in this paper.
Moreover, the interlacing property of the zeros of Jackson and Hahn-Exton $q$-Bessel
functions and their derivatives is also useful in the proof of the main results.
We also deduce a sufficient and necessary condition for the close-to-convexity of a normalized
Jackson $q$-Bessel function and its derivatives. Some open problems are
proposed at the end of the paper.
\end{abstract}

%=======================================================================================================================================================

%=======================================================================================================================================================

\section{Introduction and statements of the main results}

Let $\mathbb{D}_{r}$ be the open disk $\left\{ {z\in \mathbb{C}:\left\vert
z\right\vert <r}\right\}$ with radius $r>0$ and let $\mathbb{D}=\mathbb{D}_1.$ By $\mathcal{A}$ we mean the class of normalized analytic
functions $f:\mathbb{D}_r\to\mathbb{C}$ which satisfy the usual
normalization conditions $f(0)=f'(0)-1=0.$ Denote by $\mathcal{S}$
the class of functions belonging to $\mathcal{A}$ which are univalent in $\mathbb{D}%
_r$ and let $\mathcal{S}^{\ast }(\alpha )$ be the subclass of $\mathcal{S}$
consisting of functions which are starlike of order $\alpha $ in $\mathbb{D}%
_r,$ where $\alpha\in[0,1).$ The analytic characterization of this class of
functions is
\begin{equation*}
\mathcal{S}^{\ast }(\alpha )=\left\{ f\in \mathcal{S}\ :\  \real \left(\frac{zf'(z)}{f(z)}\right)>\alpha\ \ \mathrm{for\ all}\ \ z\in \mathbb{%
D}_r \right\},
\end{equation*}
and we adopt the convention $\mathcal{S}^{\ast}=\mathcal{S}^{\ast }(0)$. The real number
\begin{equation*}
r_{\alpha }^{\ast}(f)=\sup \left\{ r>0\ :\ \real \left(\frac{zf'(z)}{f(z)}\right)>\alpha\ \ \mathrm{ for\ all}\ \ z\in \mathbb{D}_{r}\right\},
\end{equation*}
is called the radius of starlikeness of order $\alpha $ of the function $f.$
Note that $r^{\ast }(f)=r_0^{\ast}(f)$ is the largest radius such that the
image region $f(\mathbb{D}_{r^{\ast }(f)})$ is a starlike domain with
respect to the origin.

For $\alpha\in[0,1)$ the class of convex functions of order $\alpha$ is defined by
 $$\mathcal{K}(\alpha) = \left\{f\in \mathcal{S}\ :\ \real \left(1+\frac{zf''(z)}{f'(z)}\right) > \alpha\ \ \mathrm{for\ all}\ \ z\in \mathbb{D}_r \right\},$$
 and for $\alpha=0$ it reduces to the class $\mathcal{K}$ of convex functions. We note that the convex functions do not need to be normalized, that is, the definition of $\mathcal{K}(\alpha)$ is also valid for non-normalized analytic function $f:\mathbb{D}\to\mathbb{C}$ with the property $f'(0)\neq0$, The radius of convexity of order $\alpha$ of an analytic locally
univalent function $f:\mathbb{C}\to\mathbb{C}$ is defined by
$$r^c_\alpha(f)=\sup\left\{r>0\ : \ \real \left(1+\frac{zf''(z)}{f'(z)}\right) >\alpha\ \ \mathrm{ for\ all}\ \ z\in\mathbb{D}_r\right\}.$$
We note that $r^c(f)=r^c_0(f)$ is in fact the largest radius for which the image domain $f(\mathbb{D}_{r^c(f)})$ is a convex domain in $\mathbb{C}.$ For more
information about starlike and convex functions we refer to Duren's the book \cite{duren} and to the references therein.

Consider the Jackson and Hahn-Exton $q$-Bessel functions. They are explicitly defined by
$$J_{\nu}^{(2)}(z;q)=\frac{\left(q^{\nu+1};q\right)_{\infty}}{(q;q)_{\infty}}\sum_{n\geq0}\frac{(-1)^n\left(\frac{z}{2}\right)^{2n+\nu}}{(q;q)_n\left(q^{\nu+1};q\right)_n}q^{n(n+\nu)}$$
and
$$J_{\nu}^{(3)}(z;q)=\frac{\left(q^{\nu+1};q\right)_{\infty}}{(q;q)_{\infty}}\sum_{n\geq0}\frac{(-1)^n{z}^{2n+\nu}}{(q;q)_n\left(q^{\nu+1};q\right)_n}q^{\frac{1}{2}n(n+1)},$$
where $z\in\mathbb{C},$ $\nu>-1,$ $q\in(0,1)$ and
$$(a;q)_0=1,\ \ \ (a;q)_n=\prod_{k=1}^n\left(1-aq^{k-1}\right),\ \ \ (a;q)_{\infty}=\prod_{k\geq1}\left(1-aq^{k-1}\right).$$
A common feature of these analytic functions is that they are $q$-extensions of the classical Bessel function of the first kind $J_{\nu}.$ Namely, for fixed $z$ we have $J_{\nu}^{(2)}((1-z)q;q)\to J_{\nu}(z)$ and $J_{\nu}^{(3)}((1-z)q;q)\to J_{\nu}(2z)$ as $q\nearrow 1.$ Watson's treatise \cite{Wat} contains comprehensive information about the Bessel function of the first kind and properties of the above $q$-extensions of Bessel functions cane be found in \cite{ismail,muldoon,koelink,koornwinder} and in the references therein.  Geometric properties of Bessel functions of the first kind, such as univalence, starlikeness, convexity and close-to-convexity, are established in \cite{publ,lecture,starlike,convex,close}. Motivated by those results, in this paper we study the geometric properties of the Jackson and Hahn-Exton $q$-Bessel functions. For each of them, three different normalizations are applied in such a way that the resulting functions are analytic in the unit disk of the complex plane. Using the Weierstrassian decomposition of $J_{\nu}^{(2)}$ and $J_{\nu}^{(3)}$ and combining the methods from \cite{lommel,starlike,convex} we determine precisely the radii of starlikeness and convexity for each of the six functions. These results are the $q$-generalizations of the corresponding results for Bessel functions of the first kind obtained in \cite{starlike,convex}. A characterization of entire functions from the Laguerre-P\'olya class involving hyperbolic polynomials and an interlacing property of the zeros of Jackson and Hahn-Exton $q$-Bessel functions and their derivatives play an important role in the proofs. We establish a necessary and sufficient condition for the close-to-convexity of a normalized Jackson $q$-Bessel function and its derivatives. The results obtained in the present paper show that there is no essential difference between Jackson and Hahn-Exton $q$-Bessel functions when one teats the problem about the radii of starlikeness and convexity. Therefore one may expect that other geometric properties of these two $q$-extensions of Bessel's function are also similar.

Since neither $J_{\nu}^{(2)}(\cdot;q)$, nor $J_{\nu}^{(3)}(\cdot;q)$ belongs to $\mathcal{A}$, first we perform some natural normalizations. For $\nu>-1$ we define three functions originating from $J_{\nu}^{(2)}(\cdot;q)$:
$$f_{\nu}^{(2)}(z;q)=\left(2^{\nu}c_{\nu}(q)J_{\nu}^{(2)}(z;q)\right)^{\frac{1}{\nu}},\ \nu\neq0,$$
$$g_{\nu}^{(2)}(z;q)=2^{\nu}c_{\nu}(q)z^{1-\nu}J_{\nu}^{(2)}(z;q),$$
$$h_{\nu}^{(2)}(z;q)=2^{\nu}c_{\nu}(q)z^{1-\frac{\nu}{2}}J_{\nu}^{(2)}(\sqrt{z};q),$$
where $c_{\nu}(q)=(q;q)_{\infty}/\left(q^{\nu+1};q\right)_{\infty}.$ Similarly, we associate with $J_{\nu}^{(3)}(\cdot;q)$ the functions
$$f_{\nu}^{(3)}(z;q)=\left(c_{\nu}(q)J_{\nu}^{(3)}(z;q)\right)^{\frac{1}{\nu}},\ \nu\neq0,$$
$$g_{\nu}^{(3)}(z;q)=c_{\nu}(q)z^{1-\nu}J_{\nu}^{(3)}(z;q),$$
$$h_{\nu}^{(3)}(z;q)=c_{\nu}(q)z^{1-\frac{\nu}{2}}J_{\nu}^{(3)}(\sqrt{z};q).$$
Clearly the functions $f_{\nu}^{(s)}(\cdot;q)$, $g_{\nu}^{(s)}(\cdot;q)$, $h_{\nu}^{(s)}(\cdot;q)$, $s\in\{2,3\},$ belong to the class $\mathcal{A}$.

The first principal result we establish concerns the radii of starlikeness and reads as follows.

\begin{theorem}
\label{theo1} Let $\nu>-1$ and $s\in\{2,3\}.$ The following statements hold:

\begin{enumerate}
\item[\textbf{a)}] If $\alpha\in[0,1)$ and $\nu>0,$ then $r_{\alpha }^{\ast }\left(f_{\nu}^{(s)}\right)=x_{\nu
,\alpha,1}$, where $x_{\nu,\alpha,1}$ is the smallest positive root of the
equation $$r\cdot dJ_{\nu}^{(s)}(r;q)/dr-\alpha\nu J_{\nu}^{(s)}(r;q)=0.$$ Moreover, if $\alpha\in[0,1)$ and $\nu\in(-1,0),$ then $r_{\alpha
}^{\ast }\left(f_{\nu}^{(s)}\right)=x_{\nu ,\alpha }$, where $x_{\nu ,\alpha }$ is
the unique positive root of the equation $$ir\cdot dJ_{\nu}^{(s)}(ir;q)/dr-\alpha\nu J_{\nu}^{(s)}(ir;q)=0.$$

\item[\textbf{b)}] If $\alpha\in[0,1),$ then $r_{\alpha }^{\ast }\left(g_{\nu}^{(s)}\right)=y_{\nu,\alpha,1}$, where $y_{\nu,\alpha,1}$ is the smallest positive root of the equation $$r\cdot dJ_{\nu}^{(s)}(r;q)/dr-(\alpha+\nu-1)J_{\nu}^{(s)}(r;q)=0.$$

\item[\textbf{c)}] If $\alpha\in[0,1),$ then $r_{\alpha }^{\ast }\left(h_{\nu}^{(s)}\right)=z_{\nu,\alpha,1}$, where $z_{\nu,\alpha,1}$ is the smallest positive root of the equation $$r\cdot dJ_{\nu}^{(s)}(r;q)/dr-(2\alpha+\nu-2) J_{\nu}^{(s)}(r;q)=0.$$
\end{enumerate}
\end{theorem}

Our second result concerns the radii of convexity.

\begin{theorem}
\label{theo2}
Let $\nu>-1$ and $s\in\{2,3\}.$ The following statements hold:
\begin{enumerate}
\item[{\bf a)}] If $\nu>0$ and $\alpha\in[0,1),$  then the radius of
convexity of order $\alpha$ of the function $f_\nu^{(s)}(\cdot;q)$ is the smallest positive
root of the equation
$$1+\frac{r\cdot d^2J_\nu^{(s)}(r;q)/dr^2}{dJ_\nu^{(s)}(r;q)/dr}+\left(\frac{1}{\nu}-1\right)\frac{r\cdot dJ_\nu^{(s)}(r;q)/dr}{J_\nu(r;q)}=\alpha.$$
Moreover, we have $r^{c}_{\alpha}(f_{\nu}^{(2)})<j_{\nu,1}'(q)<j_{\nu,1}(q)$ and $r^{c}_{\alpha}(f_{\nu}^{(3)})<l_{\nu,1}'(q)<l_{\nu,1}(q),$ where $j_{\nu,1}(q),$ $l_{\nu,1}(q),$ $j_{\nu,1}'(q)$ and $l_{\nu,1}'(q)$ are the first positive zeros of the functions ${J}_{\nu}^{(2)}(\cdot;q),$ ${J}_{\nu}^{(3)}(\cdot;q),$ $z\mapsto d{J}_{\nu}^{(2)}(z;q)/dz$ and $z\mapsto d{J}_{\nu}^{(3)}(z;q)/dz.$

\item[{\bf b)}] If $\nu>-1$ and $\alpha\in[0,1),$ then the radius of
convexity of order $\alpha$ of the function $g_{\nu}^{(s)}(\cdot;q)$ is the smallest positive root of
the equation
$$1-\nu+r\frac{(2-\nu)\cdot dJ_{\nu}^{(s)}(r;q)/dr+r\cdot d^2J_{\nu}^{(s)}(r;q)/dr^2}{(1-\nu)J_{\nu}^{(s)}(r;q)+r\cdot dJ_{\nu}^{(s)}(r;q)/dr}=\alpha.$$
Moreover, we have $r^{c}_{\alpha}(g_{\nu}^{(2)})<\alpha_{\nu,1}(q)<j_{\nu,1}(q)$ and $r^{c}_{\alpha}(g_{\nu}^{(3)})<\gamma_{\nu,1}(q)<l_{\nu,1}(q),$ where
$\alpha_{\nu,1}(q)$ and $\gamma_{\nu,1}(q)$ are the first positive zeros of the functions $z\mapsto z\cdot d{J}_{\nu}^{(2)}(z;q)/dz+(1-\nu){J}_{\nu}^{(2)}(z;q)$ and $z\mapsto z\cdot d{J}_{\nu}^{(3)}(z;q)/dz+(1-\nu){J}_{\nu}^{(3)}(z;q).$

\item[{\bf c)}] If $\nu>-1$ and $\alpha\in[0,1),$ then the radius of
convexity of order $\alpha$ of the function $h_{\nu}^{(s)}(\cdot;q)$ is the smallest positive root of
the equation
$$1-\frac{\nu}{2}+\frac{\sqrt{r}}{2}\frac{(3-\nu)\cdot dJ_{\nu}^{(s)}(\sqrt{r};q)/dr+\sqrt{r}\cdot d^2J_{\nu}^{(s)}(\sqrt{r};q)/dr^2}{(2-\nu)J_{\nu}^{(s)}(\sqrt{r};q)+\sqrt{r}\cdot dJ_{\nu}^{(s)}(\sqrt{r};q)/dr}=\alpha.$$
Moreover, we have $r^{c}_{\alpha}(h_{\nu}^{(2)})<\beta_{\nu,1}(q)<j_{\nu,1}(q)$ and $r^{c}_{\alpha}(h_{\nu}^{(3)})<\delta_{\nu,1}(q)<l_{\nu,1}(q),$ where
$\beta_{\nu,1}(q)$ and $\delta_{\nu,1}(q)$ are the first positive zeros of the functions $z\mapsto z\cdot d{J}_{\nu}^{(2)}(z;q)/dz+(2-\nu){J}_{\nu}^{(2)}(z;q),$  and $z\mapsto z\cdot d{J}_{\nu}^{(3)}(z;q)/dz+(2-\nu){J}_{\nu}^{(3)}(z;q).$
\end{enumerate}
\end{theorem}

We note that these theorems are natural $q$-extension to Jackson and Hahn-Exton $q$-Bessel functions of the results obtained in \cite{starlike} and \cite{convex}.
While the ideas of the proofs are similar, here we need some specific $q$-extensions of results about the Bessel functions which are of independent interest, such as
Lemmas \ref{lem1}, \ref{zerodini} and \ref{interlace}.

Finally, we state a result, which is the $q$-extension of the first part of \cite[Theorem 1]{close} for the Jackson $q$-Bessel function.

\begin{theorem}\label{theo3}
If $\nu>-1,$ then the function $h_{\nu}(\cdot;q)=h_{\nu}^{(2)}(\cdot;q)$ is starlike and all of its derivatives are close-to-convex in $\mathbb{D}$ if and only if $\nu\geq\max\{\nu_0(q),\nu^{*}(q)\},$ where $\nu_0(q)$ is the unique root of the equation $\left.dh_{\nu}^{(2)}(z;q)/dz\right|_{z=1}=h_{\nu}'(1;q)=0,$ and $\nu^{*}(q)$ is the unique root of the equation $j_{\nu,1}(q)=1.$
\end{theorem}

The paper is organized as follows. Section 2 contains the preliminary results together with their proofs, while in Section 3 we present the proofs of the main results. In Section 4 we present some consequence of the Hadamard factorizations to Rayleigh sums of the $q$-Bessel functions under discussion and formulate some open problems.

\section{Preliminaries}
\setcounter{equation}{0}

\subsection{The Hadamard factorization for $q$-Bessel functions} The following preliminary results are useful in the sequel. The next infinite product representations are natural $q$-extensions of the well-known Hadamard factorization for Bessel functions of the first kind.

\begin{lemma}
\label{lem1} If $\nu>-1,$ then $z\mapsto\mathcal{J}_{\nu}^{(2)}(z;q)=2^{\nu}c_{\nu}(q)z^{-\nu}J_{\nu}^{(2)}(z;q)$ and $z\mapsto\mathcal{J}_{\nu}^{(3)}(z;q)=c_{\nu}(q)z^{-\nu}J_{\nu}^{(3)}(z;q)$ are entire functions of order $\rho =0.$
Consequently, their Hadamard factorization for $z\in\mathbb{C}$ are of the form
\begin{equation}
\mathcal{J}_{\nu}^{(2)}(z;q)=\prod\limits_{n\geq 1}\left( 1-\frac{z^{2}}{j_{\nu,n}^{2}(q)}%
\right) , \ \
\mathcal{J}_{\nu}^{(3)}(z;q)=\prod\limits_{n\geq 1}\left( 1-\frac{z^{2}}{l_{\nu,n}^{2}(q)}%
\right) ,  \label{1.6}
\end{equation}%
where $j_{\nu,n}(q)$ and $l_{\nu,n}(q)$ are the $n$th positive zeros of the
functions ${J}_{\nu}^{(2)}(\cdot;q)$ and ${J}_{\nu}^{(3)}(\cdot;q).$
\end{lemma}

\begin{proof}[\bf Proof]
Since
$$\mathcal{J}_{\nu}^{(2)}(z;q)=2^{\nu}c_{\nu}(q)z^{-\nu}J_{\nu}^{(2)}(z;q)=\sum_{n\geq 0}\frac{(-1)^nz^{2n}q^{n(n+\nu)}}{2^{2n}(q;q)_n\left(q^{\nu+1};q\right)_n},$$
$$\mathcal{J}_{\nu}^{(3)}(z;q)=c_{\nu}(q)z^{-\nu}J_{\nu}^{(3)}(z;q)=\sum_{n\geq 0}\frac{(-1)^nz^{2n}q^{\frac{1}{2}n(n+1)}}{(q;q)_n\left(q^{\nu+1};q\right)_n},$$
it follows that the growth orders of the even entire functions $z\mapsto\mathcal{J}_{\nu}^{(2)}(z;q)$ and $z\mapsto\mathcal{J}_{\nu}^{(3)}(z;q)$ are zero. Namely, we have that
$$\lim_{n\to\infty}\frac{n\log n}{\log(q;q)_n+\log\left(q^{\nu+1};q\right)_n+2n\log2-n(n+\nu)\log q}=0,$$
$$\lim_{n\to\infty}\frac{n\log n}{\log(q;q)_n+\log\left(q^{\nu+1};q\right)_n-\frac{1}{2}n(n+1)\log q}=0,$$
since as $n\to \infty$ we have $(q;q)_n\to(q;q)_{\infty}<\infty$ and $\left(q^{\nu+1};q\right)_n\to\left(q^{\nu+1};q\right)_{\infty}<\infty.$ On the other hand, we know that the zeros $j_{\nu,n}(q),$ $n\in\mathbb{N},$ and $\l_{\nu,n}(q),$ $n\in\mathbb{N},$ are real and simple, according to \cite[Theorem 4.2]{ismail} and \cite[Theorem 3.4]{koelink}, and with this the rest of the proof of \eqref{1.6} follows by applying
Hadamard's Theorem \cite[p. 26]{levin}.  
\end{proof}

\subsection{Quotients of power series} We will also need the following result, see \cite{biernacki,pv}:

\begin{lemma}\label{lempower}
Consider the power series $f(x)=\displaystyle\sum_{n\geq 0}a_{n}x^n$ and $g(x)=\displaystyle\sum_{n\geq 0}b_{n}x^n$,
where $a_{n}\in \mathbb{R}$ and $b_{n}>0$ for all $n\geq 0$. Suppose that both series converge on $(-r,r)$, for some $r>0$. If the
sequence $\lbrace a_n/b_n\rbrace_{n\geq 0}$ is increasing (decreasing), then the function $x\mapsto{f(x)}/{g(x)}$ is increasing
(decreasing) too on $(0,r)$. The result remains true for the power series
$$f(x)=\displaystyle\sum_{n\geq 0}a_{n}x^{2n}\ \ \ \mbox{and}\ \ \ g(x)=\displaystyle\sum_{n\geq 0}b_{n}x^{2n}.$$
\end{lemma}

\subsection{Zeros of polynomials and entire functions, and the Laguerre-P\'olya class} In this subsection
we provide the necessary information about polynomials and entire functions with real zeros. An algebraic polynomial
is called hyperbolic if all its zeros are real. We note that the simple statement that two real polynomials $p$ and $q$ posses real and interlacing zeros if and only if any linear combinations of $p$ and $q$ is a hyperbolic polynomial is sometimes called Obrechkoff's theorem. We formulate the following specific statement that we shall need, see \cite{lommel}.
\begin{lemma} \label{OLem}
Let $p(x)=1-a_1 x +a_2 x^2 -a_3 x^3 + \cdots +(-1)^n a_n x^n = (1-x/x_1)\cdots (1-x/x_n)$ be a hyperbolic polynomial with positive zeros
$0< x_1\leq x_2 \leq \cdots \leq x_n$, and normalized by $p(0)=1$. Then, for any constant $C$, the polynomial $q(x) = C p(x) - x\, p'(x)$ is hyperbolic. Moreover, the smallest
zero $\eta_1$   belongs to the interval $(0,x_1)$ if and only if $C<0$.
\end{lemma}

The proof of this result is straightforward; it is enough to apply Rolle's theorem and then count the sign changes of the linear combination at the zeros of $p$. We refer to \cite{BDR, DMR} for further results on monotonicity and asymptotics of zeros of linear combinations of hyperbolic polynomials.

A real entire function $\psi$ belongs to the Laguerre-P\'{o}lya class $\lp$ if it can be represented in the form
$$
\psi(x) = c x^{m} e^{-a x^{2} + \beta x} \prod_{k\geq1}
\left(1+\frac{x}{x_{k}}\right) e^{-\frac{x}{x_{k}}},
$$
with $c,$ $\beta,$ $x_{k} \in \mathbb{R},$ $a \geq 0,$ $m\in
\mathbb{N} \cup\{0\},$ $\sum x_{k}^{-2} < \infty.$
Similarly, $\phi$  is  said to be of
type I in the Laguerre-P\'{o}lya class, written $\varphi \in \lpi$,
if $\phi(x)$ or $\phi(-x)$ can be represented as
$$
\phi(x) = c x^{m} e^{\sigma x} \prod_{k\geq1}\left(1+\frac{x}{x_{k}}\right),
$$
with $c \in \mathbb{R},$ $\sigma \geq 0,$ $m \in
\mathbb{N}\cup\{0\},$ $x_{k}>0,$ $\sum 1/x_{k} < \infty.$
The class $\lp$ is the complement of the space of hyperbolic
polynomials in the topology induced by the uniform convergence
on the compact sets of the complex plane while $\lpi$ is the complement
of the hyperbolic polynomials whose zeros posses a preassigned constant sign.
Given an entire function $\varphi$ with the Maclaurin expansion
$$\varphi(x) = \sum_{k\geq0}\gamma_{k} \frac{x^{k}}{k!},$$
its Jensen polynomials are defined by
$$
g_n(\varphi;x) = g_{n}(x) = \sum_{k=0}^{n} {n\choose k} \g_{k}
x^k.
$$
Jensen  proved the following relation in \cite{Jen12}:

\begin{lemma}\label{JTh}
The function $\varphi$ belongs to $\lp$ ($\lpi$, respectively) if and only if
all the polynomials $g_n(\varphi;x)$, $n\in\mathbb{N}$, are hyperbolic (hyperbolic
with zeros of equal sign).
Moreover, the sequence $g_n(\varphi;z/n)$ converges locally
uniformly to $\varphi(z)$.
\end{lemma}

Further information about the Laguerre-P\'olya class can be found in
\cite{Obr, RS} while \cite{DC} contains references  and additional facts
about the Jensen polynomials in general and also about those related to the
Bessel function.

The following result is a key tool in the proof of Theorems \ref{theo1} and \ref{theo2}.

\begin{lemma}\label{ThZ} Let $\nu>-1$ and $a<0$.
Then the functions $z\mapsto (2a+\nu)J_{\nu}^{(2)}(z;q)-z\cdot dJ_{\nu}^{(2)}(z;q)/dz$ and $z\mapsto (2a+\nu)J_{\nu}^{(3)}(z;q)-z\cdot dJ_{\nu}^{(3)}(z;q)/dz$
can be represented in the form
$$c_{\nu}(q)\left( (2a+\nu)J_{\nu}^{(2)}(z;q)-z\cdot dJ_{\nu}^{(2)}(z;q)/dz \right) = 2\left(\frac{z}{2}\right)^{\nu} \phi_\nu(z;q),$$
$$c_{\nu}(q)\left( (2a+\nu)J_{\nu}^{(3)}(z;q)-z\cdot dJ_{\nu}^{(3)}(z;q)/dz \right) = 2{z}^{\nu} \psi_\nu(z;q),$$
where $\phi_\nu(\cdot;q)$ and $\psi_\nu(\cdot;q)$ are entire functions which belongs to the Laguerre-P\'olya class $\lp$. Moreover, the smallest positive zero of $\phi_\nu(\cdot;q)$ does not exceed the first positive zero $j_{\nu,1}(q),$ while the smallest positive zero of $\psi_\nu(\cdot;q)$ is less than $l_{\nu,1}(q)$.
\end{lemma}

\begin{proof}[\bf Proof]
It is clear from the infinite product representation of $z\mapsto \mathcal{J}_{\nu}^{(2)}(z;q)=2^{\nu}c_{\nu}(q)z^{-\nu}J_{\nu}^{(2)}(z;q)$ that this function belongs to the Laguerre-P\'olya class of entire functions (since the exponential factors in the infinite product are canceled because of the symmetry of the zeros $\pm j_{\nu,n}(q),$ $n\in\mathbb{N},$ with respect to the origin). This implies that the function $z\mapsto \mathcal{J}_{\nu}^{(2)}(2\sqrt{z};q)=\tilde{\mathcal{J}}_{\nu}(z;q)$ belongs to $\lpi.$ Then it follows form Lemma \ref{JTh} that its Jensen polynomials
$$g_n(\tilde{\mathcal{J}}_{\nu}(\cdot;q);\zeta) = \sum_{k=0}^n  {n\choose k} \frac{k!}{\left(q;q\right)_k \left(q^{\nu+1};q\right)_k}q^{k(\nu+k)}\left(-\zeta \right)^{k}$$
are all hyperbolic. However, observe that the Jensen polynomials of $\tilde{\phi}_\nu(z;q)= \phi_\nu(2\sqrt{z};q)$ are simply
$$
g_n(\tilde{\phi}_\nu;\zeta)=a g_n(\tilde{\mathcal{J}}_{\nu}(\cdot;q);\zeta) - \zeta\, g_n'(\tilde{\mathcal{J}}_{\nu}(\cdot;q);\zeta).
$$
Lemma \ref{OLem} implies that all zeros of $g_n(\tilde{\phi}_\nu;\zeta)$ are real and positive and that the smallest one
precedes the first zero of  $g_n(\tilde{\mathcal{J}}_{\nu}(\cdot;q);\zeta)$. In view of Lemma \ref{JTh}, the latter conclusion immediately yields that $\tilde{\phi}_\nu \in \lpi$ and that its first zero precedes $j_{\nu,1}(q)$. Finally, the first part of the statement of the lemma follows after we go back
from  $\tilde{\phi}_\nu(\cdot;q)$ to $\phi_\nu(\cdot;q)$ by setting $\zeta=\frac{z^2}{4}$.

Similarly, because of Lemma \ref{lem1} the function $z\mapsto \mathcal{J}_{\nu}^{(3)}(z;q)=c_{\nu}(q)z^{-\nu}J_{\nu}^{(3)}(z;q)$ belongs to the Laguerre-P\'olya class of entire functions, which implies that the function $z\mapsto \mathcal{J}_{\nu}^{(3)}(\sqrt{z};q)=\overline{\mathcal{J}}_{\nu}(z;q)$ belongs to $\lpi.$ Then it follows from Lemma \ref{JTh} that its Jensen polynomials
$$g_n(\overline{\mathcal{J}}_{\nu}(\cdot;q);\zeta) = \sum_{k=0}^n  {n\choose k} \frac{k!}{\left(q;q\right)_k \left(q^{\nu+1};q\right)_k}q^{\frac{1}{2}k(k+1)}\left(-\zeta \right)^{k}$$
are all hyperbolic. However, observe that the Jensen polynomials of $\tilde{\psi}_\nu(z;q)= \psi_\nu(\sqrt{z};q)$ are simply
$$
g_n(\tilde{\psi}_\nu;\zeta)=a g_n(\overline{\mathcal{J}}_{\nu}(\cdot;q);\zeta) - \zeta\, g_n'(\overline{\mathcal{J}}_{\nu}(\cdot;q);\zeta).
$$
Lemma \ref{OLem} implies that all zeros of $g_n(\tilde{\psi}_\nu;\zeta)$ are real and positive and that the smallest one
precedes the first zero of  $g_n(\overline{\mathcal{J}}_{\nu}(\cdot;q);\zeta)$. In view of Lemma \ref{JTh}, the latter conclusion immediately yields that $\tilde{\psi}_\nu \in \lpi$ and that its first zero precedes $l_{\nu,1}(q)$. Thus, the second part of the statement of this lemma follows after we go back
from  $\tilde{\psi}_\nu(\cdot;q)$ to $\psi_\nu(\cdot;q)$ by setting $\zeta={z^2}$.
\end{proof}

The following result is an immediate consequence of Lemma \ref{ThZ} and is the $q$-extension to Jackson and Hahn-Exton $q$-Bessel functions of the well known result that if $\nu>-1$ and $c$ is a constant such that $c+\nu>0,$ then the Dini function $z\mapsto zJ_{\nu}'(z)+cJ_{\nu}(z)$ has only real zeros and its first positive zero does not exceed the first positive zero of $J_{\nu},$ see \cite[p. 597]{Wat} and \cite[p. 11]{muld}.

\begin{lemma}\label{zerodini}
 If $\nu>-1$ and $c$ is a constant such that $c+\nu>0,$ then the Jackson $q$-Dini function $z\mapsto z\cdot dJ_{\nu}^{(2)}(z;q)/dz+cJ_{\nu}^{(2)}(z;q)$ has only real zeros and its first positive zero does not exceed $j_{\nu,1}(q).$ Similarly, under the same assumptions the Hahn-Exton $q$-Dini function $z\mapsto z\cdot dJ_{\nu}^{(3)}(z;q)/dz+cJ_{\nu}^{(3)}(z;q)$ has only real zeros and its first positive zero does not exceed $l_{\nu,1}(q).$
\end{lemma}

\subsection{The Hadamard factorization of the derivatives of $q$-Bessel functions} The following infinite product representations are natural $q$-extensions of the well-known Hadamard factorization for the derivative of Bessel functions of the first kind.

\begin{lemma}
\label{lemder} If $\nu>0,$ then $z\mapsto (2^{\nu}/\nu)c_{\nu}(q)z^{1-\nu}\cdot dJ_{\nu}^{(2)}(z;q)/dz$ and $z\mapsto (1/\nu)c_{\nu}(q)z^{1-\nu}\cdot dJ_{\nu}^{(3)}(z;q)/dz$ are entire functions of order $\rho =0.$
Consequently, their Hadamard factorization for $z\in\mathbb{C}$ are of the form
\begin{equation}
dJ_{\nu}^{(2)}(z;q)/dz=\frac{\nu\left(\frac{1}{2}z\right)^{\nu-1}}{2c_{\nu}(q)}\prod\limits_{n\geq 1}\left( 1-\frac{z^{2}}{j_{\nu,n}'^{2}(q)}%
\right) , \ \
dJ_{\nu}^{(3)}(z;q)/dz=\frac{\nu z^{\nu-1}}{c_{\nu}(q)}\prod\limits_{n\geq 1}\left( 1-\frac{z^{2}}{l_{\nu,n}'^{2}(q)}%
\right) ,  \label{1.6der}
\end{equation}%
where $j_{\nu,n}'(q)$ and $l_{\nu,n}'(q)$ are the $n$th positive zeros of $z\mapsto d{J}_{\nu}^{(2)}(z;q)/dz$ and $z\mapsto d{J}_{\nu}^{(3)}(z;q)/dz.$
\end{lemma}

\begin{proof}[\bf Proof]
We have that
$$\frac{1}{\nu}2^{\nu}c_{\nu}(q)z^{1-\nu}\cdot dJ_{\nu}^{(2)}(z;q)/dz=\frac{1}{\nu}\sum_{n\geq 0}\frac{(-1)^n(2n+\nu)z^{2n}q^{n(n+\nu)}}{2^{2n}(q;q)_n\left(q^{\nu+1};q\right)_n},$$
$$\frac{1}{\nu}c_{\nu}(q)z^{1-\nu}\cdot dJ_{\nu}^{(3)}(z;q)/dz=\frac{1}{\nu}\sum_{n\geq 0}\frac{(-1)^n(2n+\nu)z^{2n}q^{\frac{1}{2}n(n+1)}}{(q;q)_n\left(q^{\nu+1};q\right)_n},$$
and
$$\lim_{n\to\infty}\frac{n\log n}{\log(q;q)_n+\log\left(q^{\nu+1};q\right)_n+2n\log2-n(n+\nu)\log q-\log(2n+\nu)}=0,$$
$$\lim_{n\to\infty}\frac{n\log n}{\log(q;q)_n+\log\left(q^{\nu+1};q\right)_n-\frac{1}{2}n(n+1)\log q-\log(2n+\nu)}=0,$$
since as $n\to \infty$ we have $(q;q)_n\to(q;q)_{\infty}<\infty$ and $\left(q^{\nu+1};q\right)_n\to\left(q^{\nu+1};q\right)_{\infty}<\infty.$ Moreover, we know that the zeros $j_{\nu,n}'(q),$ $n\in\mathbb{N},$ and $\l_{\nu,n}'(q),$ $n\in\mathbb{N},$ are real for $\nu>0$, according to Lemma \ref{zerodini}, and with this the rest of the proof of \eqref{1.6der} follows by applying Hadamard's Theorem \cite[p. 26]{levin}.  
\end{proof}

\subsection{The Hadamard factorization of $q$-Dini functions} The following infinite product representations are $q$-extensions of the known Hadamard factorization for Dini functions $z\mapsto zJ_{\nu}'(z)+(1-\nu)J_{\nu}(z)$ and $z\mapsto zJ_{\nu}'(z)+(2-\nu)J_{\nu}(z),$ see \cite[Theorem 1]{monotonicity}.

\begin{lemma}
\label{dini} If $\nu>-1,$ then $z\mapsto dg_{\nu}^{(2)}(z;q)/dz,$ $z\mapsto dh_{\nu}^{(2)}(z;q)/dz,$ $z\mapsto dg_{\nu}^{(3)}(z;q)/dz$ and $z\mapsto dh_{\nu}^{(3)}(z;q)/dz$ are entire functions of order $\rho =0.$ Consequently, their Hadamard factorization for $z\in\mathbb{C}$ are of the form
$$
dg_{\nu}^{(2)}(z;q)/dz=\prod\limits_{n\geq 1}\left( 1-\frac{z^{2}}{\alpha_{\nu,n}^{2}(q)}%
\right) , \ \
dh_{\nu}^{(2)}(z;q)/dz=\prod\limits_{n\geq 1}\left( 1-\frac{z}{\beta_{\nu,n}^{2}(q)}%
\right) ,  $$
$$
dg_{\nu}^{(3)}(z;q)/dz=\prod\limits_{n\geq 1}\left( 1-\frac{z^{2}}{\gamma_{\nu,n}^{2}(q)}%
\right) , \ \
dh_{\nu}^{(3)}(z;q)/dz=\prod\limits_{n\geq 1}\left( 1-\frac{z}{\delta_{\nu,n}^{2}(q)}%
\right) ,  $$
where $\alpha_{\nu,n}(q)$ and $\beta_{\nu,n}(q)$ are the $n$th positive zeros of $z\mapsto z\cdot d{J}_{\nu}^{(2)}(z;q)/dz+(1-\nu){J}_{\nu}^{(2)}(z;q)$ and $z\mapsto z\cdot d{J}_{\nu}^{(2)}(z;q)/dz+(2-\nu){J}_{\nu}^{(2)}(z;q),$ while $\gamma_{\nu,n}(q)$ and $\delta_{\nu,n}(q)$ are the $n$th positive zeros of $z\mapsto z\cdot d{J}_{\nu}^{(3)}(z;q)/dz+(1-\nu){J}_{\nu}^{(3)}(z;q)$ and $z\mapsto z\cdot d{J}_{\nu}^{(3)}(z;q)/dz+(2-\nu){J}_{\nu}^{(3)}(z;q).$
\end{lemma}

\begin{proof}[\bf Proof]
We have that
$$\frac{dg_{\nu}^{(2)}(z;q)}{dz}=2^{\nu}c_{\nu}(q)z^{-\nu}\left(z\cdot \frac{d{J}_{\nu}^{(2)}(z;q)}{dz}+(1-\nu){J}_{\nu}^{(2)}(z;q)\right)=\sum_{n\geq 0}\frac{(-1)^n(2n+1)z^{2n}q^{n(n+\nu)}}{2^{2n}(q;q)_n\left(q^{\nu+1};q\right)_n},$$
$$\frac{dh_{\nu}^{(2)}(z;q)}{dz}=2^{\nu-1}c_{\nu}(q)z^{-\nu/2}\left(\sqrt{z}\cdot \frac{d{J}_{\nu}^{(2)}(\sqrt{z};q)}{dz}+(2-\nu){J}_{\nu}^{(2)}(\sqrt{z};q)\right)=\sum_{n\geq 0}\frac{(-1)^n(n+1)z^{n}q^{n(n+\nu)}}{2^{2n}(q;q)_n\left(q^{\nu+1};q\right)_n},$$
$$\frac{dg_{\nu}^{(3)}(z;q)}{dz}=c_{\nu}(q)z^{-\nu}\left(z\cdot \frac{d{J}_{\nu}^{(3)}(z;q)}{dz}+(1-\nu){J}_{\nu}^{(3}(z;q)\right)=\sum_{n\geq 0}\frac{(-1)^n(2n+1)z^{2n}q^{\frac{1}{2}n(n+1)}}{(q;q)_n\left(q^{\nu+1};q\right)_n},$$
$$\frac{dh_{\nu}^{(3)}(z;q)}{dz}=\frac{1}{2}c_{\nu}(q)z^{-\nu/2}\left(\sqrt{z}\cdot \frac{d{J}_{\nu}^{(3)}(\sqrt{z};q)}{dz}+(2-\nu){J}_{\nu}^{(3}(\sqrt{z};q)\right)=\sum_{n\geq 0}\frac{(-1)^n(n+1)z^{n}q^{\frac{1}{2}n(n+1)}}{(q;q)_n\left(q^{\nu+1};q\right)_n},$$
and
$$\lim_{n\to\infty}\frac{n\log n}{\log(q;q)_n+\log\left(q^{\nu+1};q\right)_n+2n\log2-n(n+\nu)\log q-\log(2n+1)}=0,$$
$$\lim_{n\to\infty}\frac{n\log n}{\log(q;q)_n+\log\left(q^{\nu+1};q\right)_n+2n\log2-n(n+\nu)\log q-\log(n+1)}=0,$$
$$\lim_{n\to\infty}\frac{n\log n}{\log(q;q)_n+\log\left(q^{\nu+1};q\right)_n-\frac{1}{2}n(n+1)\log q-\log(2n+1)}=0,$$
$$\lim_{n\to\infty}\frac{n\log n}{\log(q;q)_n+\log\left(q^{\nu+1};q\right)_n-\frac{1}{2}n(n+1)\log q-\log(n+1)}=0,$$
since as $n\to \infty$ we have $(q;q)_n\to(q;q)_{\infty}<\infty$ and $\left(q^{\nu+1};q\right)_n\to\left(q^{\nu+1};q\right)_{\infty}<\infty.$ Moreover, we know that the zeros $\alpha_{\nu,n}(q),$ $\beta_{\nu,n}(q),$ $\gamma_{\nu,n}(q),$ $\delta_{\nu,n}(q),$ $n\in\mathbb{N},$  are real for $\nu>-1$, according to Lemma \ref{zerodini}, and with this the rest of the proof follows by applying Hadamard's Theorem \cite[p. 26]{levin}.  
\end{proof}

\subsection{Interlacing of zeros of $q$-Bessel functions and their derivatives} The next result complements the other interlacing properties of the zeros of Jackson and Hahn-Exton $q$-Bessel functions, see \cite[Theorem 4.3]{ismail} and \cite[Theorem 3.7]{koelink}. This preliminary result is necessary in the proof of the first part of Theorem \ref{theo2}.

\begin{lemma}\label{interlace}
Between any two consecutive roots of the function $z\mapsto J_{\nu}^{(s)}(z;q)$ the function $z\mapsto dJ_{\nu}^{(s)}(z;q)/dz$ has precisely one zero when $\nu>-1$ and $s\in\{2,3\}.$
\end{lemma}

\begin{proof}[\bf Proof]
The proofs for Jackson $q$-Bessel and Hahn-Exton $q$-Bessel are very similar, and thus we will give the proof only for $s=2.$ For the simplicity in this proof we will use the notation $\mathcal{J}_{\nu}(z;q)=2^{\nu}c_{\nu}(q)z^{-\nu}J_{\nu}(z;q)$ instead of $\mathcal{J}_{\nu}^{(2)}(z;q)=2^{\nu}c_{\nu}(q)z^{-\nu}J_{\nu}^{(2)}(z;q).$ Moreover, we will use simply $J_{\nu}'(z;q)$ instead of $dJ_{\nu}^{(2)}(z;q)/dz.$ Since $\mathcal{J}_{\nu}(\cdot;q)$ belongs to the Laguerre-P\'olya class of entire functions, it follows that it satisfies the Laguerre inequality (see \cite{skov})
\begin{equation}\label{laguerre}\left(\mathcal{J}_{\nu}(z;q)^{(n)}\right)^2-\left(\mathcal{J}_{\nu}(z;q)\right)^{(n-1)}\left(\mathcal{J}_{\nu}(z;q)\right)^{(n+1)}>0,\end{equation}
where $\nu>-1$ and $z\in\mathbb{R}.$ On the other hand, we have that
$$\mathcal{J}_{\nu}'(z;q)=2^{\nu}c_{\nu}(q)z^{-\nu-1}\left(zJ_{\nu}'(z;q)-\nu J_{\nu}(z;q)\right),$$
$$\mathcal{J}_{\nu}''(z;q)=2^{\nu}c_{\nu}(q)z^{-\nu-2}\left(z^2J_{\nu}''(z;q)-2\nu zJ_{\nu}'(z;q)+\nu(\nu+1)J_{\nu}(z;q)\right),$$
and thus the Laguerre inequality \eqref{laguerre} for $n=1$ is equivalent to
$$2^{2\nu}c_{\nu}(q)z^{-2\nu-2}\left(z^2\left(J_{\nu}'(z;q)\right)^2-z^2J_{\nu}(z;q)J_{\nu}''(z;q)-\nu J_{\nu}^2(z;q)\right)>0.$$
This implies that $$\left(J_{\nu}'(z;q)\right)^2-J_{\nu}(z;q)J_{\nu}''(z;q)>\nu J_{\nu}^2(z;q)/z^2>0$$
for $\nu>0$ and $z\in\mathbb{R},$ that is, the function $z\mapsto J_{\nu}'(z)/J_{\nu}(z)$ is decreasing on $(0,\infty)\setminus\left\{\left.j_{\nu,n}(q)\right|n\in\mathbb{N}\right\}.$ Recall that the zeros $j_{\nu,n}(q),$ $n\in\mathbb{N},$ of the Jackson $q$-Bessel function are real and simple, according to \cite[Theorem 4.2]{ismail}, and thus $J_{\nu}'(z;q)$ does not vanish in $j_{\nu,n}(q),$ $n\in\mathbb{N}.$ Thus, for a fixed $k\in\mathbb{N}$ the function $z\mapsto J_{\nu}'(z)/J_{\nu}(z)$ takes the limit $\infty$ when $z\searrow j_{\nu,k-1}(q),$ and the limit $-\infty$ when $z\nearrow j_{\nu,k}(q).$ Moreover, since $z\mapsto J_{\nu}'(z)/J_{\nu}(z)$ is decreasing on $(0,\infty)\setminus\left\{\left.j_{\nu,n}(q)\right|n\in\mathbb{N}\right\}$ it results that in each interval $(j_{\nu,k-1}(q),j_{\nu,k}(q))$ its restriction intersects the horizontal line only once, and the abscissa of this intersection point is exactly $j_{\nu,k}'(q).$ Here we used the convention that $j_{\nu,0}(q)=0.$
\end{proof}

\subsection{Starlikeness of entire functions in the open unit disk}
The next result (see \cite[Theorem 2]{st}) is the key tool in the proof of Theorem \ref{theo3}.

\begin{lemma}\label{lemshah}
Let $f:\mathbb{D}\to\mathbb{C}$ be a transcendental entire function of the form
$$f(z)=z\prod_{n\geq 1}\left(1-\frac{z}{z_n}\right),$$
where all $z_n$ have the same argument and satisfy $|z_n|>1.$ If $f$ is univalent in $\mathbb{D},$ then
\begin{equation}\label{ineqshah}
\sum_{n\geq1}\frac{1}{|z_n|-1}\leq 1.
\end{equation}
In fact \eqref{ineqshah} holds if and only if $f$ is starlike in $\mathbb{D}$ and all of its derivatives are
close-to-convex there.
\end{lemma}

\section{Proofs of the main results}
\setcounter{equation}{0}

\begin{proof}[\bf Proof of Theorem \ref{theo1}]
The proofs for the cases $s=2$ and $s=3$ are almost the same, the only difference is that we have different zeros $j_{\nu,n}(q)$ and $l_{\nu,n}(q)$ in the proofs. Thus, we will present the proof only for the case $s=2$ and in what follows for the simplicity we will use the following notations: $J_{\nu}(z;q)=J_{\nu}^{(2)}(z;q),$ $f_{\nu}(z;q)=f_{\nu}^{(2)}(z;q),$ $g_{\nu}(z;q)=g_{\nu}^{(2)}(z;q),$ $h_{\nu}(z;q)=h_{\nu}^{(2)}(z;q)$ and $J_{\nu}'(z;q)=dJ_{\nu}^{(2)}(z;q)/dz.$

First we prove part {\bf a} for $\nu>0$ and parts {\bf b} and {\bf c} for $\nu>-1.$ We need to show that for the corresponding values of $\nu$ and $\alpha$ the
inequalities
\begin{equation}
\real \left( \frac{zf_{\nu }'(z;q)}{f_{\nu }(z;q)}\right) >\alpha ,%
\text{ \ \ }\real \left( \frac{zg_{\nu }'(z;q)}{g_{\nu }(z;q)}\right)
>\alpha \text{ \ and \ }\real \left( \frac{zh_{\nu }'(z;q)}{h_{\nu
}(z;q)}\right) >\alpha \text{ \ \ }  \label{2.0}
\end{equation}%
are valid for $z\in \mathbb{D}_{r_{\alpha }^{\ast }(f_{\nu })}$, $z\in
\mathbb{D}_{r_{\alpha }^{\ast }(g_{\nu })}$ and $z\in \mathbb{D}_{r_{\alpha
}^{\ast }(h_{\nu })}$ respectively, and each of the above inequalities does
not hold in larger disks. It follows from (\ref{1.6}) that%
$$\frac{zf'_{\nu}(z;q)}{f_{\nu}(z;q)}=\frac{1}{\nu}\frac{zJ'_{\nu}(z;q)}{J_{\nu}(z;q)}=1-\frac{1}{\nu}\sum_{n\geq1}\frac{2z^{2}}{j_{\nu,n}^{2}(q)-z^{2}},$$
$$\frac{zg'_{\nu}(z;q)}{g_{\nu}(z;q)}=1-\nu+\frac{zJ'_{\nu}(z;q)}{J_{\nu}(z;q)}=1-\sum_{n\geq1}\frac{2z^{2}}{j_{\nu,n}^{2}(q)-z^{2}}$$
and
$$\frac{zh'_{\nu}(z;q)}{h_{\nu}(z;q)}=1-\frac{\nu}{2}+\frac{1}{2}\frac{\sqrt{z}J'_{\nu}(\sqrt{z};q)}{J_{\nu}(\sqrt{z};q)}=1-\sum_{n\geq1}\frac{z}{j_{\nu,n}^{2}(q)-z}.$$
On the other hand, it is known that (see \cite{starlike}) if ${z\in \mathbb{C}}$ and $\beta $ ${\in \mathbb{R}}$ are such that $\beta >{\left\vert
z\right\vert }$, then%
\begin{equation}
\frac{{\left\vert z\right\vert }}{\beta -{\left\vert z\right\vert }}\geq
\real\left( \frac{z}{\beta -z}\right) .  \label{2.5}
\end{equation}%
Then the inequality
$$
\frac{{\left\vert z\right\vert }^{2}}{j_{\nu ,n}^{2}(q)-{\left\vert
z\right\vert }^{2}}\geq\real\left( \frac{z^{2}}{j_{\nu,n}^{2}(q)-z^{2}}%
\right),
$$
holds for every $\nu>-1$, $n\in \mathbb{N}$ and ${\left\vert z\right\vert <}j_{\nu,1}(q),$ which in turn implies that
$$\real\left(\frac{zf'_{\nu}(z;q)}{f_{\nu}(z;q)}\right)=1-\frac{1}{\nu}\real\left(\sum_{n\geq1}\frac{2z^{2}}{j_{\nu,n}^{2}(q)-z^{2}}\right)\geq
1-\frac{1}{\nu}\sum_{n\geq1}\frac{2|z|^{2}}{j_{\nu,n}^{2}(q)-|z|^{2}}=\frac{|z|f'_{\nu}(|z|;q)}{f_{\nu}(|z|;q)},$$
$$\real\left(\frac{zg'_{\nu}(z;q)}{g_{\nu}(z;q)}\right)=1-\real\left(\sum_{n\geq1}\frac{2z^{2}}{j_{\nu,n}^{2}(q)-z^{2}}\right)\geq
1-\sum_{n\geq1}\frac{2|z|^{2}}{j_{\nu,n}^{2}(q)-|z|^{2}}=\frac{|z|g'_{\nu}(|z|;q)}{g_{\nu}(|z|;q)}
$$
and
$$\real\left(\frac{zh'_{\nu}(z;q)}{h_{\nu}(z;q)}\right)=1-\real\left(\sum_{n\geq1}\frac{z}{j_{\nu,n}^{2}(q)-z}\right)\geq
1-\sum_{n\geq1}\frac{|z|}{j_{\nu,n}^{2}(q)-|z|}=\frac{|z|h'_{\nu}(|z|;q)}{h_{\nu}(|z|;q)},$$
with equality when $z=|z|=r$. The latter  inequalities and
the minimum principle for harmonic functions imply that the
corresponding inequalities in (\ref{2.0}) hold if and only if $%
\left\vert z\right\vert <x_{\nu ,\alpha,1},$ $\left\vert z\right\vert <y_{\nu
,\alpha,1}$ and $\left\vert z\right\vert <z_{\nu ,\alpha,1},$ respectively,
where $x_{\nu ,\alpha,1 }$, $y_{\nu ,\alpha,1 }$ and $z_{\nu,\alpha,1}$ are the
smallest positive roots of the equations%
\begin{equation*}
rf_{\nu}'(r;q)/f_{\nu}(r;q)=\alpha ,\text{ \ }rg_{\nu}'(r;q)/g_{\nu}(r;q)=\alpha ,\ rh_{\nu}'(r;q)/h_{\nu}(r;q)=\alpha.
\end{equation*}%
Since their solutions coincide with the zeros of the functions
$$r\mapsto rJ_{\nu}'(r;q)-\alpha\nu J_{\nu}(r;q),\ r\mapsto rJ_{\nu}'(r;q)-(\alpha+\nu-1)J_{\nu}(r;q),\ r\mapsto rJ_{\nu}'(r;q)-(2\alpha+\nu-2) J_{\nu}(r;q),$$
the result we need follows from Lemma \ref{ThZ} by taking instead of $a$ the values $\frac{1}{2}(\alpha-1)\nu,$ $\frac{1}{2}(\alpha-1)$ and $a=\alpha-1,$ respectively. In other words, Lemma \ref{ThZ} show that all the zeros of the above three functions are real and their first positive zeros do not exceed the first positive zero $j_{\nu,1}(q)$. This guarantees that the above inequalities hold. This completes the proof of part {\bf a} when $\nu>0$, and parts {\bf b} and {\bf c} when $\nu>-1.$

Now, to prove the statement for part {\bf a} when $\nu \in (-1,0)$ we use the counterpart of (\ref{2.5}), that is,
\begin{equation}
\real\left( \frac{z}{\beta -z}\right) \geq \frac{-{\left\vert
z\right\vert }}{\beta +{\left\vert z\right\vert }},  \label{2.10}
\end{equation}%
which holds for all ${z\in \mathbb{C}}$ and $\beta $ ${\in \mathbb{R}}$ such
that $\beta >{\left\vert z\right\vert }$ (see \cite{starlike}). From (\ref%
{2.10}), we obtain the inequality
$$
\real\left( \frac{z^{2}}{j_{\nu,n}^{2}(q)-z^{2}}\right) \geq \frac{-{%
\left\vert z\right\vert }^{2}}{j_{\nu,n}^{2}(q)+{\left\vert z\right\vert }%
^{2}},  \label{2.11}
$$
which holds for all $\nu>-1,$ $n\in \mathbb{N}$
and ${\left\vert z\right\vert <}j_{\nu,1}(q)$ and it implies that%
\begin{equation*}
\real\left(\frac{zf_{\nu}'(z;q)}{f_{\nu}(z;q)}\right) =1-%
\frac{1}{\nu}\real\left(\sum\limits_{n\geq1}\frac{%
2z^{2}}{j_{\nu,n}^{2}(q)-z^{2}}\right) \geq 1+\frac{1}{\nu}%
\sum\limits_{n\geq1}\frac{2\left\vert z\right\vert ^{2}}
{j_{\nu,n}^{2}(q)+\left\vert z\right\vert ^{2}}=\frac{i\left\vert z\right\vert f_{\nu}'(i\left\vert z\right\vert;q)}{f_{\nu}(i\left\vert
z\right\vert;q)}.
\end{equation*}%
In this case equality is attained if $z=i\left\vert z\right\vert =ir.$ Moreover, the latter inequality implies that
\begin{equation*}
\real\left( \frac{zf_{\nu}'(z;q)}{f_{\nu}(z;q)}\right) >\alpha
\end{equation*}
if and only if $\left\vert z\right\vert <x_{\nu,\alpha }$, where $x_{\nu,\alpha }$ denotes the smallest positive root of the equation $$irf_{\nu}'({i}r;q)/f_{\nu}({i}r;q)=\alpha,$$ which is equivalent to
\begin{equation*}
irJ_{\nu}'(ir;q)-\alpha\nu J_{\nu}(ir;q)=0,\text{ for }\nu \in (-1,0).
\end{equation*}%
It follows from Lemma \ref{ThZ} that the first positive zero of $z\mapsto izJ_{\nu}'(iz;q)-\alpha\nu J_{\nu}(iz;q)$ does not exceed $j_{\nu,1}(q)$ which guarantees that the above inequalities are valid. All we need to prove is that the above function has actually only one zero in $(0,\infty)$. Observe that, according to Lemma \ref{lempower}, the function
$$r\mapsto \frac{irJ_{\nu}'(ir;q)}{J_{\nu}(ir;q)}=\left.\sum_{n\geq0}\frac{(2n+\nu)\left(\frac{r}{2}\right)^{2n+\nu}}{(q;q)_n\left(q^{\nu+1};q\right)_n}q^{n(n+\nu)}
\right/\sum_{n\geq0}\frac{\left(\frac{r}{2}\right)^{2n+\nu}}{(q;q)_n\left(q^{\nu+1};q\right)_n}q^{n(n+\nu)}$$
is increasing on $(0,\infty)$ as a quotient of two power series whose positive coefficients form the increasing ``quotient sequence'' $\left\{2n+\nu\right\}_{n\geq0}.$ On the other hand, the above function tends to $\nu$ when $r\to0,$ so that its graph can intersect the horizontal line $y=\alpha\nu>\nu$ only once. This completes the proof of part {\bf a} of the theorem when $\nu\in(-1,0)$.
\end{proof}

\begin{proof}[\bf Proof of Theorem \ref{theo2}]
The proofs for the cases $s=2$ and $s=3$ are almost the same, the only difference is that we have different zeros in the proofs. Thus, as in the previous proof, we will present the proof only for the case $s=2$ and in what follows for the simplicity we will use the following notations: $J_{\nu}(z;q)=J_{\nu}^{(2)}(z;q),$ $f_{\nu}(z;q)=f_{\nu}^{(2)}(z;q),$ $g_{\nu}(z;q)=g_{\nu}^{(2)}(z;q),$ $h_{\nu}(z;q)=h_{\nu}^{(2)}(z;q),$ $J_{\nu}'(z;q)=dJ_{\nu}^{(2)}(z;q)/dz$ and $J_{\nu}''(z;q)=d^2J_{\nu}^{(2)}(z;q)/dz^2.$

{\bf a)} Since
$$1+\frac{zf_{\nu}''(z;q)}{f'_{\nu}(z;q)}=1+\frac{zJ_\nu''(z;q)}{J_\nu'(z;q)}+\left(\frac{1}{\nu}-1\right)\frac{zJ_\nu'(z;q)}{J_\nu(z;q)}.$$
and by means of \eqref{1.6der} we have
$$\frac{zJ'_\nu(z;q)}{J_\nu(z;q)}=\nu-\sum_{n\geq1}\frac{2z^2}{j_{\nu,n}^2(q)-z^2},\ \ 1+\frac{zJ_\nu''(z;q)}{J_\nu'(z;q)}=\nu-\sum_{n\geq1}\frac{2z^2}{j_{\nu,n}'^2(q)-z^2},$$
it follows that
$$1+\frac{zf_{\nu}''(z;q)}{f_{\nu}'(z;q)}=1-\left(\frac{1}{\nu}-1\right)\sum_{n\geq1}\frac{2z^2}{j_{\nu,n}^2(q)-z^2}
-\sum_{n\geq1}\frac{2z^2}{j_{\nu,n}'^2(q)-z^2}.$$
Now, suppose that $\nu\in(0,1].$ By using the inequality \eqref{2.5}, for all $z\in\mathbb{D}_{j_{\nu,1}'(q)}$ we obtain the inequality
$$\real\left(1+\frac{zf_{\nu}''(z;q)}{f_{\nu}'(z;q)}\right)\geq1-\left(\frac{1}{\nu}-1\right)
\sum_{n\geq1}\frac{2r^2}{j_{\nu,n}^2(q)-r^2}-\sum_{n\geq1}\frac{2r^2}{j_{\nu,n}'^2(q)-r^2},$$
where $|z|=r.$ Moreover, observe that if we use the inequality \cite[Lemma 2.1]{convex}
$$\lambda\real\left(\frac{z}{a-z}\right)-\real\left(\frac{z}{b-z}\right)\geq\lambda\frac{|z|}{a-|z|}-\frac{|z|}{b-|z|},$$
where $a>b>0,$ $\lambda\in[0,1]$ and $z\in\mathbb{C}$ such that $|z|<b,$ then we get that the above inequality is also valid when $\nu>1.$ Here we used that the zeros $j_{\nu,n}(q)$ and $j_{\nu,n}'(q)$ interlace according to Lemma \ref{interlace}. The above inequality implies for $r\in(0,j_{\nu,1}'(q))$
$$\inf_{z\in\mathbb{D}_r}\left\{\real\left(1+\frac{zf_{\nu}''(z;q)}{f_{\nu}'(z;q)}\right)\right\}=1+\frac{rf_{\nu}''(r;q)}{f_{\nu}'(r;q)}.$$
On the other hand, the function $u_{\nu}(\cdot;q):(0,j_{\nu,1}'(q))\to\mathbb{R},$ defined by $u_{\nu}(r;q)=1+{rf_{\nu}''(r;q)}/{f_{\nu}'(r;q)},$ is  strictly decreasing since \begin{align*}\frac{du_{\nu}(r;q)}{dr}&=-\left(\frac{1}{\nu}-1\right)\sum_{n\geq1}\frac{4rj_{\nu,n}^2(q)}{(j_{\nu,n}^2(q)-r^2)^2}
-\sum_{n\geq1}\frac{4rj_{\nu,n}'^2(q)}{(j_{\nu,n}'^2-r^2)^2}\\&<\sum_{n\geq1}\frac{4rj_{\nu,n}^2(q)}{(j_{\nu,n}^2(q)-r^2)^2}
-\sum_{n\geq1}\frac{4rj_{\nu,n}'^2(q)}{(j_{\nu,n}'^2(q)-r^2)^2}<0\end{align*}
for $\nu>0$ and $r\in(0,j_{\nu,1}'(q)).$ Here we used again that the zeros $j_{\nu,n}(q)$ and $j_{\nu,n}'(q)$ interlace and for all $n\in\mathbb{N},$ $\nu>0$ and $r<\sqrt{j_{\nu,1}(q)j_{\nu,1}'(q)}$ we have that $$j_{\nu,n}^2(q)(j_{\nu,n}'^2(q)-r^2)^2<j_{\nu,n}'^2(q)(j_{\nu,n}^2(q)-r^2)^2.$$ Since $\lim_{r\searrow0}u_{\nu}(r;q)=1>\alpha$ and $\lim_{r\nearrow j_{\nu,1}'(q)}u_{\nu}(r;q)=-\infty,$ it follows that for $z\in\mathbb{D}_{r_1}$ we have
$$\real\left(1+\frac{zf_{\nu}''(z;q)}{f_{\nu}'(z;q)}\right)>\alpha$$
if and only if  $r_1$ is the unique root of $$1+\frac{rf_{\nu}''(r;q)}{f_{\nu}'(r;q)}=\alpha$$ situated in $(0,j_{\nu,1}'(q)).$

{\bf b)} In view of Lemma \ref{dini} we have that
$$1+z\frac{g''_{\nu}(z;q)}{g'_{\nu}(z;q)}=1-\sum_{n\geq1}\frac{2z^2}{\alpha_{\nu,n}^2(q)-z^2}.$$
and by using the inequality \eqref{2.5} we obtain that
$$\real\left(1+z\frac{g''_{\nu}(z;q)}{g'_{\nu}(z;q)}\right)\geq1-\sum_{n\geq1}\frac{2r^2}{\alpha_{\nu,n}^2(q)-r^2},$$
where $|z|=r.$ Thus, for $r\in(0,\alpha_{\nu,1}(q))$ we get
$$\inf_{z\in\mathbb{D}_r}\left\{\real\left(1+\frac{zg_{\nu}''(z;q)}{g_{\nu}'(z;q)}\right)\right\}=
1-\sum_{n\geq1}\frac{2r^2}{\alpha_{\nu,n}^2(q)-r^2}=1+\frac{rg_{\nu}''(r;q)}{g_{\nu}'(r;q)}.$$
The function
$v_{\nu}(\cdot;q):(0,\alpha_{\nu,1}(q))\to\mathbb{R},$ defined by
$v_{\nu}(r;q)=1+{rg_{\nu}''(r;q)}/{g_{\nu}'(r;q)},$ is strictly decreasing
and $\lim_{r\searrow0}v_{\nu}(r;q)=1,$ $\lim_{r\nearrow\alpha_{\nu,1}(q)}v_{\nu}(r;q)=-\infty.$
Consequently for $z\in\mathbb{D}_{r_2}$ we have that
$$\real\left(1+\frac{zg_{\nu}''(z;q)}{g_{\nu}'(z;q)}\right)>\alpha$$
if and only if $r_2$ is the unique root of $$1+\frac{rg_{\nu}''(r;q)}{g_{\nu}'(r;q)}=\alpha$$ situated in $(0,\alpha_{\nu,1}(q)).$
Finally, the inequality $\alpha_{\nu,1}(q)<j_{\nu,1}(q)$ follows from Lemma \ref{zerodini}.

{\bf c)} In view of Lemma \ref{dini} we have that
$$1+z\frac{h''_{\nu}(z;q)}{h'_{\nu}(z;q)}=1-\sum_{n\geq1}\frac{2z}{\beta_{\nu,n}^2(q)-z}.$$
and by using the inequality \eqref{2.5} we obtain that
$$\real\left(1+z\frac{h''_{\nu}(z;q)}{h'_{\nu}(z;q)}\right)\geq1-\sum_{n\geq1}\frac{2r}{\beta_{\nu,n}^2(q)-r},$$
where $|z|=r.$ Thus, for $r\in(0,\alpha_{\nu,1}^2(q))$ we get
$$\inf_{z\in\mathbb{D}_r}\left\{\real\left(1+\frac{zh_{\nu}''(z;q)}{h_{\nu}'(z;q)}\right)\right\}=
1-\sum_{n\geq1}\frac{2r}{\beta_{\nu,n}^2(q)-r}=1+\frac{rh_{\nu}''(r;q)}{h_{\nu}'(r;q)}.$$
The function
$w_{\nu}(\cdot;q):(0,\beta_{\nu,1}^2(q))\to\mathbb{R},$ defined by
$w_{\nu}(r;q)=1+{rh_{\nu}''(r;q)}/{h_{\nu}'(r;q)},$ is strictly decreasing
and $\lim_{r\searrow0}w_{\nu}(r;q)=1,$ $\lim_{r\nearrow\beta_{\nu,1}^2(q)}w_{\nu}(r;q)=-\infty.$
Consequently for $z\in\mathbb{D}_{r_3}$ we have that
$$\real\left(1+\frac{zh_{\nu}''(z;q)}{h_{\nu}'(z;q)}\right)>\alpha$$
if and only if  $r_3$ is the unique root of $1+{rh_{\nu}''(r;q)}/{h_{\nu}'(r;q)}=\alpha$ situated in $(0,\beta_{\nu,1}^2(q)).$
Finally, the inequality $\beta_{\nu,1}(q)<j_{\nu,1}(q)$ follows from Lemma \ref{zerodini}, and this completes the proof of this theorem.
\end{proof}

\begin{proof}[\bf Proof of Theorem \ref{theo3}]
Observe that
$$h_{\nu}(z;q)=z\prod_{n\geq1}\left(1-\frac{z}{j_{\nu,n}^2(q)}\right).$$
Since $\nu\mapsto j_{\nu,n}(q)$ is increasing for each $n\in\mathbb{N},$ $q\in(0,1)$ on $(-1,\infty),$ see \cite[Theorem 3]{muldoon}, it follows that
$j_{\nu,1}(q)>1$ if and only if $\nu>\nu^{*}(q),$ where $\nu^{*}(q)$ is the unique root of $j_{\nu,1}(q)=1,$ and consequently $j_{\nu,n}(q)>1$ for all $n\in\mathbb{N}$ and $\nu>\nu^{*}(q).$ On the other hand, from the above infinite representation we have
$$\frac{zh_{\nu}'(z;q)}{h_{\nu}(z;q)}=1-\sum_{n\geq 1}\frac{z}{j_{\nu}^2(q)-z},$$
which implies that
$$\frac{h_{\nu}'(1;q)}{h_{\nu}(1;q)}=1-\sum_{n\geq 1}\frac{1}{j_{\nu}^2(q)-1},$$
and this greater or equal than zero if and only if $\nu\geq\nu_0(q),$ where $\nu_0(q)$ is the unique root of the equation
$h_{\nu}'(1;q)=0.$ Here we used that for all $\nu>-1$ we have
$$\frac{d}{d\nu}\left(\frac{f_{\nu}'(1;q)}{f_{\nu}(1;q)}\right)=\sum_{n\geq1}\frac{2j_{\nu,n}(q){ dj_{\nu,n}(q)}/{d\nu}}{(j_{\nu,n}^2(q)-1)^2}\geq0,$$
since the function $\nu\mapsto j_{\nu,n}(q)$ is increasing on $(-1,\infty)$ for all fixed $n\in\mathbb{N}.$ Thus, applying Lemma \ref{lemshah}, the conclusion of this theorem follows immediately.
\end{proof}

\begin{figure}[!ht]
\centering
\includegraphics[width=11.5cm]{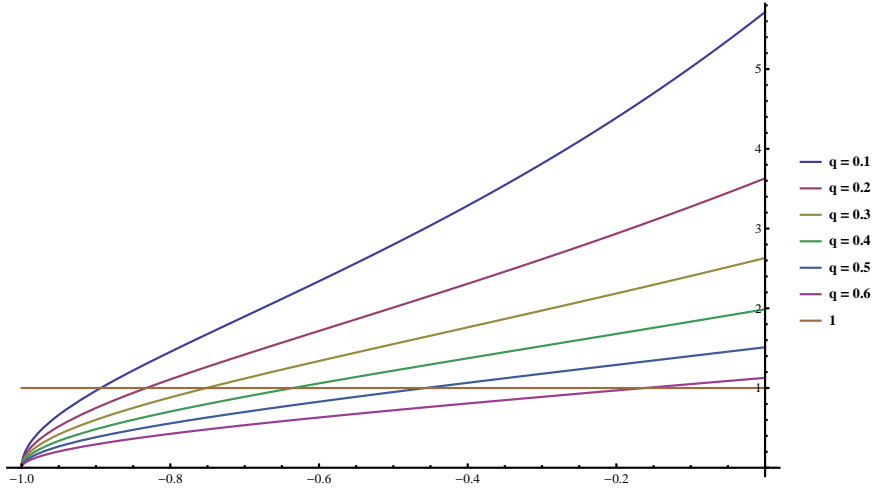}
\caption{The graph of the function $\nu\mapsto j_{\nu,1}(q)$ for $q\in\{0.1,\dots,0.6\}$ and of the horizontal line $y=1$ on $(-1,0).$}
\label{fig1}
\end{figure}

\section{Further results and concluding remarks}
\setcounter{equation}{0}

In this section we present some consequence of the Hadamard factorizations related to the Rayleigh sum of the $q$-Bessel functions under discussion and formulate some open problems.
\begin{enumerate}
\item[\bf A.] By using the infinite series and product representations of $h_{\nu}^{(2)}(\cdot;q),$ that is,
$$z\prod_{n\geq 1}\left(1-\frac{z}{j_{\nu,n}^2(q)}\right)=\sum_{n\geq 0}\frac{(-1)^nz^{n+1}}{2^{2n}(q;q)_n\left(q^{\nu+1};q\right)_n}q^{n(\nu+n)},$$
and comparing the coefficients of $z^2$ on the both sides of the above equation, it follows that
$$\sum_{n\geq1}\frac{1}{j_{\nu,n}^2(q)}=\frac{q^{\nu+1}}{4(q-1)(q^{\nu+1}-1)}.$$
This is actually the $q$-extension of the well-known first Rayleigh sum for Bessel functions of the first kind. Multiplying by $(1-q)^2$ both sides of the above equation and tending with $q$ to $1^{-}$ we obtain the Rayleigh sum in the question, that is,
$$\sum_{n\geq1}\frac{1}{j_{\nu,n}^2}=\frac{1}{4(\nu+1)},$$
where $j_{\nu,n}$ stands for the $n$th positive zero of the classical Bessel function $J_{\nu}.$ Now, taking into account that $j_{\nu,n}^2(q)-1<j_{\nu,n}^2(q)$ for all $\nu>-1,$ $q\in(0,1)$ and $n\in\mathbb{N},$ it follows that
$$\frac{h_{\nu}'(1;q)}{h_{\nu}(1;q)}=1-\sum_{n\geq 1}\frac{1}{j_{\nu}^2(q)-1}>1-\sum_{n\geq 1}\frac{1}{j_{\nu}^2(q)}=\frac{4q^{\nu+2}-5q^{\nu+1}-4q+4}{4(q-1)(q^{\nu+1}-1)}=\tau_{\nu}(q).$$
The figures \ref{fig1} and \ref{fig2} of $\nu\mapsto j_{\nu,1}(q)$ and $\nu\mapsto \tau_{\nu}(q)$ for some fixed values of $q$ suggest that in Theorem \ref{theo3} we have that $\max\{\nu^{*}(q),\nu_0(q)\}=\nu_0(q),$ exactly as in the case of the classical Bessel functions of the first kind. However, we were unable to prove that indeed $\max\{\nu^{*}(q),\nu_0(q)\}=\nu_0(q).$

\item[\bf B.] Here we present the $q$-analogue of another lower order Rayleigh sums of the Bessel function $J_{\nu}$. Note that the Rayleigh sums are sums of the form \[\sigma_\nu(2m)=\sum_{n\geq 1}\frac{1}{j_{\nu,n}^{2m}},\] and it is known that
\begin{equation}
\sigma_\nu(4)=\frac{1}{16(\nu+1)^2(\nu+2)}.\label{Rayleighsigma}
\end{equation}
See \cite{Wat} for more information on $\sigma_\nu(2m)$ in general. Let
\[\sigma_\nu^{(2)}(2m;q)=\sum_{n\geq 1}\frac{1}{j_{\nu,n}^{2m}(q)},\quad\mbox{and}\quad\sigma_\nu^{(3)}(2m;q)=\sum_{n\geq1}\frac{1}{l_{\nu,n}^{2m}(q)}.\]
According to the limit relations $J_\nu^{(2)}((1-q)z;q)\to J_\nu(z)$ and $J_\nu^{(3)}((1-q)z;q)\to J_\nu(2z)$, clearly we have that
$$\lim_{q\nearrow1}(1-q)^{2m}\sigma_\nu^{(2)}(2m;q)=\sigma_\nu(2m),\ \ \ \mbox{and}\ \ \ \lim_{q\nearrow1}\frac{(1-q)^{2m}}{2^{2m}}\sigma_\nu^{(3)}(2m;q)=\sigma_\nu(2m).$$
Now, we are going to determine the low-order $q$-Rayleigh sums $\sigma_\nu^{(s)}(2;q)$ and $\sigma_\nu^{(s)}(4;q)$ for $s\in\{2,3\}$. Earlier we saw in the previous comment that
\[\sigma_\nu^{(2)}(2;q)=\frac{q^{1+\nu}}{4(1-q)(1-q^{1+\nu})}.\]
In order to determine $\sigma_\nu^{(2)}(4;q)$ we apply
\[h_\nu^{(2)}(z;q)h_\nu^{(2)}(-z;q)=-z^2\prod_{n\ge1}\left(1-\frac{z^2}{j_{\nu,n}^4(q)}\right).\]
Comparing the coefficients of $z^4$ on both sides we have
\[\sigma_\nu^{(2)}(4;q)=\frac{q^{2\nu+2}}{16 (1-q)^2 \left(1-q^{\nu+1}\right)^2}-\frac{q^{2\nu+4}}{8 (1-q)(1-q^2)(1-q^{nu+1})(1-q^{\nu+2})}.\]
One can verify that
\[(1-q)^4\sigma_\nu^{(2)}(4;q)\to\frac{1}{16(\nu+1)^2(\nu+2)},\]
which agrees with \eqref{Rayleighsigma}.
Similarly, for the roots of the Hahn-Exton $q$-Bessel function
$$\sigma_\nu^{(3)}(2;q)=\frac{q}{(1-q)(1-q^{1+\nu})},\ \ \sigma_\nu^{(3)}(4;q)=\frac{q^2}{(1-q)^2 \left(1-q^{\nu+1}\right)^2}-\frac{2 q^3}{(1-q)(1-q^2)(1-q^{\nu+1})(1-q^{\nu+2})}.$$

Finally, we note that the classical Rayleigh inequalities \cite[p. 502]{Wat}
\[\left(\sigma_\nu(2m)\right)^{-{1}/{m}}<j_{\nu,1}^2<\frac{\sigma_\nu(2m)}{\sigma_\nu(2m+2)}\]
can easily be transferred to
\[\left(\sigma_\nu^{(2)}(2m;q)\right)^{-{1}/{m}}<j_{\nu,1}^2(q)<\frac{\sigma_\nu^{(2)}(2m;q)}{\sigma_\nu^{(2)}(2m+2;q)},\]
and
\[\left(\sigma_\nu^{(3)}(2m;q)\right)^{-{1}/{m}}<l_{\nu,1}^2(q)<\frac{\sigma_\nu^{(3)}(2m;q)}{\sigma_\nu^{(3)}(2m+2;q)}.\]
These relations come from the facts that the zeros are real and are ordered such that $j_{\nu,n}(q)<j_{\nu,n+1}(q)$ and $l_{\nu,n}(q)<l_{\nu,n+1}(q)$ for all $n\in\mathbb{N}.$

\item[\bf C.] We note that many other known results on Bessel functions of the first kind, like complete monotonicity properties and inequalities can be extended to Jackson and Hahn-Exton $q$-Bessel functions by using the infinite product representations presented in this paper. Moreover, it would be also interesting to see the monotonicity of the zeros of $z\mapsto z\cdot dJ_{\nu}^{(s)}(z;q)/dz+cJ_{\nu}^{(s)}(z;q),$ $s\in\{2,3\}$ with respect to the order $\nu$ and to extend Theorem \ref{theo3} to the functions $f_{\nu}^{(s)}(\cdot;q)$ and $g_{\nu}^{(s)}(\cdot;q),$ $s\in\{2,3\},$ as well as to consider the convexity of these functions together with $h_{\nu}^{(s)}(\cdot;q)$ in the open unit disk $\mathbb{D}.$ Moreover, we would like to see also the counterpart of Theorem \ref{theo3} for $h_{\nu}^{(3)}(\cdot;q).$ These problems may be of interest for further research.
\end{enumerate}

\begin{figure}[!ht]
\centering
\includegraphics[width=11.5cm]{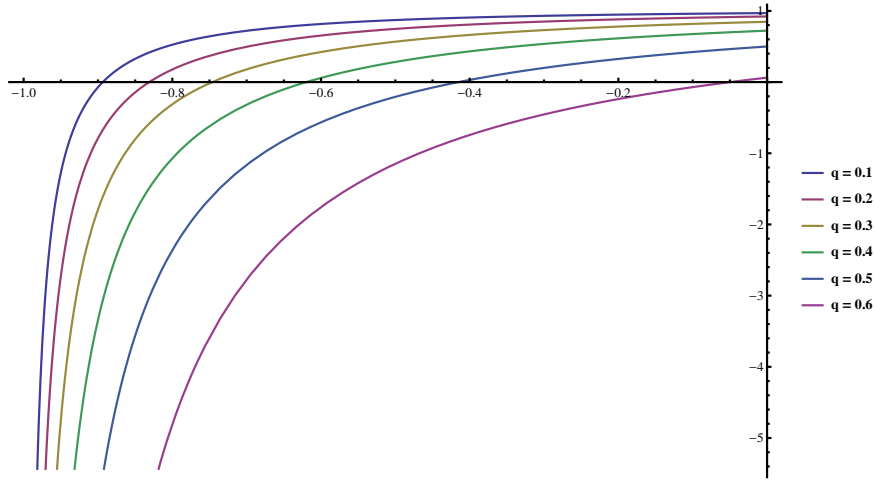}
\caption{The graph of the function $\nu\mapsto \tau_{\nu}(q)$ for $q\in\{0.1,\dots,0.6\}$ on $(-1,0).$}
\label{fig2}
\end{figure}

\end{document}